\documentclass[11pt]{article}

\usepackage{geometry}
\usepackage{bbding,textcomp,amsfonts,amsthm,amsmath,amssymb,graphicx,color,mathtools}
\usepackage{hyperref}
\usepackage[capitalize]{cleveref}
\usepackage{xcolor}
\usepackage[utf8]{inputenc}
\usepackage[T1]{fontenc}
\title{On the chromatic number of powers\\ of subdivisions of graphs}
\author{Michael Anastos\thanks{Institute of Science and Technology Austria (ISTA), Klosterneurburg 3400, Austria. Email: {\tt michael.anastos@ist.ac.at}}, Simona Boyadzhiyska\thanks{School of Mathematics, University of Birmingham, Edgbaston, Birmingham, B15 2TT, UK.
Email: {\tt s.s.boyadzhiyska@bham.ac.uk}}, Silas Rathke\thanks{Fachbereich Mathematik und Informatik, Freie Universität Berlin, Arnimallee 3, 14195 Berlin, Germany. Email: \texttt{s.rathke@fu-berlin.de}}, Juanjo Ru\'e\thanks{Departament de Matem\`atiques, Universitat Polit\`ecnica de Catalunya and Centre de Recerca Matemàtica (CRM), Spain. Email: {\tt juan.jose.rue@upc.edu}} }
\date{}

\newtheorem{theorem}{Theorem}[section]
\newtheorem{claim}[theorem]{Claim}
\newtheorem{lemma}[theorem]{Lemma}

\newcommand{\subs}{\subseteq}

\newcommand{\X}{\chi}
\renewcommand{\(}{\!\left(}
\renewcommand{\)}{\!\right)}
\newcommand{\G}{G^{\frac{2}{3}}}
\newcommand{\D}{\Delta}
\newcommand{\coleq}{\coloneqq}
\newcommand{\abs}[1]{\left\lvert#1\right\rvert}

\usepackage{enumitem}
\setlist[enumerate,1]{label={\textnormal{(\roman*)}}}

\DeclareMathOperator{\dst}{dst}
\DeclarePairedDelimiter{\parens}{(}{)}
\DeclarePairedDelimiter{\set}{\{}{\}}

\begin{document}

\maketitle

\begin{abstract}
For a given graph $G=(V,E)$, we define its \emph{$n$th subdivision} as the graph obtained from $G$ by replacing every edge by a path of length $n$. We also define the \emph{$m$th power} of $G$ as the graph on vertex set $V$ where we connect every pair of vertices at distance at most $m$ in $G$. In this paper, we study the chromatic number of powers of subdivisions of graphs and   resolve the case $m=n$ asymptotically. In particular, our result confirms a conjecture  of Mozafari-Nia and Iradmusa in the case $m=n=3$ in a strong sense. 
\end{abstract}

\section{Introduction}

Let $G=(V,E)$ be a simple graph. A \emph{total colouring} of $G$ is an assignment of colours to its vertices and edges so that no pair of adjacent vertices or edges has the same colour, and no edge has the same colour as either of its endpoints. We denote by $\chi''(G)$ (called the \textit{total chromatic number} of $G$) the minimum number of colours needed in a total colouring of $G$.

The \emph{total colouring conjecture}, posed independently by Vizing in $1964$ \cite{vizing64} and by Behzad~\cite{behzad1965graphs} in his Ph.D.\ dissertation in $1965$, states that, for every simple graph $G$ with maximum degree $\Delta(G)$, we have $\chi''(G)\leq \Delta(G)+2$. 
Nowadays, there are partial advances towards this conjecture. For example, Reed and Molloy proved in \cite{molloyReed98} that, if $\Delta(G)$ is sufficiently large, then
$\chi''(G)\leq \Delta(G)+C$, where $C$ can be taken to be $10^{26}$ (however, the authors state that the constant is not optimised and a detailed analysis could yield a much better constant). Hind, Reed, and Molloy proved in \cite{hindmolloyreed99} that, if $\Delta(G)$ is sufficiently large, then $\chi''(G)\leq \Delta(G)+ 8 (\log\Delta(G))^8$.  We refer the reader to~\cite{survey,MolloyReed,yap} for some history and further results in this line of research.

In this paper, we study generalisations of the total colouring conjecture. 
For a given graph $G=(V,E)$, we define its \emph{$n$th subdivision} as the graph obtained from $G$ by replacing every edge with a path of length $n$. We also define the $m$th power of $G$ as the graph on vertex set $V$ where we connect every pair of vertices at distance\footnote{The \emph{length} of a path $P$ is the number of edges of $P$. The \emph{distance} between two vertices $u$ and $v$ is the minimum length of a $u$-$v$-path.} at most $m$ in $G$. We denote by $G^{\frac{1}{n}}$ and $G^m$ the $n$th subdivision and the $m$th power of $G$, respectively. 
Finally, we define $G^{\frac{m}{n}}$ to be $\left(G^{\frac{1}{n}}\right)^{m}$. For instance, $G^{\frac{1}{1}}=G$. In $G^{\frac{m}{n}}$, the vertices that were already in $G$ are called \emph{branch vertices}, whereas the vertices that were added because of the subdivision are called \emph{inner vertices}.

Note that $\X(G^{\frac{1}{n}})\le 3$ holds for all $n>1$ and all graphs $G$ because we can always assign colour 1 to the branch vertices and then alternatingly colour the inner vertices of the subdivision with the colours 2 and 3.
Observe also that $\chi''(G)=\chi(G^{\frac{2}{2}})$; hence, the total colouring conjecture states that $\chi(G^{\frac{2}{2}}) \leq \Delta(G)+2$. For other powers of subdivisions, it was shown by Iradmusa~\cite[Lemma~3]{iradmusa2010colorings}, and later also by Hartke, Liu, and Pet{\v{r}}{\'\i}{\v{c}}kov{\'a}~\cite[Lemma~2.6]{hartke2012coloring}, that $\chi(G^{\frac{2}{3}}) = \Delta(G)+1$ whenever $\Delta(G)\geq 3$. More generally, the latter group of authors studied the chromatic number of $G^{\frac{m}{n}}$ when $1<m<n$ and determined it up to an additive constant of 2 (see \cite[Theorems~1 and~2]{hartke2012coloring} and \cite[Theorem~1]{iradmusa2010colorings}).  The case $m=n$ is much less understood.
 Wang and Liu proved in \cite{wang2018coloring} that, if $\Delta(G)\leq 3$, then $\chi(G^{\frac{3}{3}})\leq 7$. This was extended by Mozafari-Nia and Iradmusa in \cite{mozafari2021note}, who showed that, if $\Delta(G)\leq 4$, then $\chi(G^{\frac{3}{3}})\leq 9$, and conjectured that $\chi(G^{\frac{3}{3}})\leq 2\Delta(G)+1$. In~\cite{mozafari2022simultaneous,mozafari2023outerplanar,mozafari2024hypercubes}, the same authors checked the validity of this conjecture  for some classes of graphs such as $k$-degenerate graphs, cycles, forests, complete graphs, hypercubes, outerplanar graphs, and regular bipartite
graphs.

Our main result proves this conjecture in  asymptotic form (when $\Delta\rightarrow \infty $), and shows that the multiplicative constant can in fact be taken to be 1.

\begin{theorem}\label{theo:X of G33}
There is a constant $C>0$ such that, for all graphs $G$ with maximum degree~$\D$, \[\X\parens{G^{\frac{3}{3}}}\le \Delta+C\log\Delta.\]
\end{theorem}

The constant $C$ in \cref{theo:X of G33} can be taken  to be $28$ when $\Delta$ is sufficiently large, although in our argument we have not attempted to optimise it.
The proof uses some ideas from Guiduli~\cite{Guiduli} as well as Alon, McDiarmid, and Reed~\cite{alon1992star}. In the first paper, the author considers the incidence colouring number $\iota(G)$, which can be seen as the colouring number where one only has to colour the inner vertices of $G^{\frac{3}{3}}$. We show that the proof can be extended by probabilistic means to colour the branch vertices as well.

Concerning lower bounds for this problem, Guiduli~\cite{Guiduli} also notes that, if $G$ is a Paley graph, then $\iota(G)\ge \Delta+\Omega(\log\Delta)$. This implies that, for infinitely many $\D$, there are graphs $G$ with maximum degree $\D$ such that
\[\X\parens{G^{\frac{3}{3}}}\ge\D+\Omega(\log \D).\]

As a byproduct of the ideas developed to prove Theorem \ref{theo:X of G33}, we are also able to obtain the following generalisation for the chromatic number of $G^{\frac{k}{k}}$ when $k\geq 2$:

\begin{theorem}\label{thm:fin}
For every integer $k\geq 2$, there exists a constant $C_k$ such that, for every graph $G$, we have:
$$\bigg\lfloor \frac{k}{2} \bigg\rfloor \Delta(G)\leq \chi(G^{\frac{k}{k}}) \leq  \bigg\lfloor \frac{k}{2}\bigg \rfloor \Delta(G) + C_k\log \Delta(G).$$
\end{theorem}

This short paper is organised as follows. In \cref{sec:thm1.2}, we recall basic notions concerning the directed linear arboricity of a graph and prove \cref{theo:X of G33}. We sketch the proof of \cref{thm:fin} and discuss the difficulties arising when trying to generalise our method to fractions $\frac{r}{s}$ where $r>s$ in \cref{s:thm1.3}.

\renewcommand{\G}{G^{\frac{3}{3}}}
\section{The chromatic number of $\G$: Proof of \Cref{theo:X of G33}} \label{sec:thm1.2}

The proof of \Cref{theo:X of G33} follows the arguments from the proof of Theorem 3.1 in \cite{Guiduli}. There, the author considers the  \emph{directed star arboricity} $\dst(D)$ of a directed graph $D$, defined as the smallest number of directed star forests needed to cover $D$, where the edges of the star are directed away from the centre. The directed star arboricity is closely connected to the \emph{incidence colouring number} $\iota(G)$ of a graph $G$. This is the smallest number of colours needed to colour the vertex-edge pairs $(v,e)$, with $v\in e$, of $G$ in such a way that $(v,e)$ and $(w,f)$ receive different colours if $v=w$ or $vw=e$ or $vw=f$. 
By viewing the pair $(v,e)$ as an orientation of $e$ towards $v$, we observe the following connection between these two notions: If $S(G)$ is the directed graph where each edge of $G$ is replaced by both directed edges, then $\iota(G)=\dst(S(G))$. In \cite{Guiduli}, Guiduli showed that $\dst(D)\le k+20\log k+84$, where $k$ is the larger of the maximum indegree and the maximum outdegree of $D$.

For an edge $e=vw$ in $G$, we write $e^v$ for the neighbour of $v$  in $G^{\frac{1}{3}}$ on the subdivision of~$e$.
In other words, the edge $e=vw$ in $G$ defines the path $v, e^v,e^w,w$ in $G^{\frac{1}{3}}$. 
In $\G$, the vertices $e^v, e^w$ are inner vertices, whereas the vertices $v$ and $w$ are branch vertices. 
For a branch vertex $v$, let $I_v\coleq\lbrace e^v:e\in E(G)\rbrace$ be the set of all inner vertices which are neighbours of $v$ in $G^{\frac{1}{3}}$. Note that  $\{I_v\}_{v\in V(G)}$ partitions the set of inner vertices.

If we identify the inner vertex $e^v$ of $\G$ with the incidence pair $(v,e)$, we can quickly see that $\iota(G)$ is the colouring number of $\G$ when we only have to colour the inner vertices. In particular, $\X\parens{ \G } \ge \iota(G)$. We can easily complete a colouring of the inner vertices to a colouring of all of $\G$ by using $\chi(G)\leq \Delta(G)+1$ additional colours. Our result shows that we can in fact accomplish this task with only logarithmically many (in $\D$) additional colours. Hence, \Cref{theo:X of G33} is a slight generalisation of Theorem 3.1 in~\cite{Guiduli}.

We will follow the proof of \cite{Guiduli} and adjust it in such a way that it becomes apparent that we can also colour the branch vertices of $\G$.
The proof uses a version of the Lovász Local Lemma (Lemma 3.4 in \cite{Guiduli}), which we restate here for completeness.

\begin{lemma}\label{LLL}
    Let $H$ be a simple graph on vertex set $V=[n]$ with maximum degree $\Delta(H)\le d$. A probability event $A_i$ is associated with each vertex $i\in V$ such that $\Pr(A_i)\le 1/(4d)$ and the event $A_i$ is independent of all $A_j$ for which $j$ is not adjacent to $i$ in $H$. Then $\Pr(\bar{A_i}\cap\dots\cap\bar{A_n})>0$.
\end{lemma}

Given a graph $G$ with maximum degree $\Delta$, our general strategy for colouring $\X\parens{\G}$ will be the following:
\newcommand{\mref}[1]{\hyperref[Step #1]{#1)}}
\begin{enumerate}
    \item[1)]\label{Step 1} Find a proper colouring $c$ of the branch vertices of $\G$ using colours in $[\Delta+1]$.
    \item[2)]\label{Step 2} For each branch vertex $v$ of $\G$, find a list $L_v\subseteq [\Delta+C \log \Delta]$ of colours ($C$ will be defined explicitly later) of size $\Theta(\log\D)$. 
    For a branch vertex $w$ and an edge $e=wv$ of $G$, we will colour the inner vertex $e^w$ with a colour from the list $L_v$. For that, we will require the  
    family $\{L_v\}_{v\in V(G)}$ to have the following property: for each branch vertex $w$, $\big\lbrace L_v\setminus\lbrace c(v),c(w)\rbrace\big\rbrace_{v\in N_G(w)}$ has a transversal\footnote{A \emph{transversal} of a family $S_1,\dots,S_m$ of sets consists of $m$ distinct elements $x_1,\dots,x_m$ such that $x_i\in S_i$.} $T_w$.
    \item[3)]\label{Step 3} Colour the inner vertices of $\G$ around $w$ according to the transversal $T_w$ obtained in Step \mref2. 
\end{enumerate}

At this point of the colouring process, the only monochromatic edges that can occur are of the form $e^vf^w$, where $e=vw$ and we allow $f=e$. But this can only happen if $L_w$ contains the colour of $f^w$ (note that the colour of $e^v$ comes from $L_w$). Thus, for each branch vertex $v$, there are at most $\Theta(\log\D)$ many inner vertices of the form $e^v$ that have to be recoloured. As a final step, we use a small number of additional colours to resolve the conflicts: 

\begin{enumerate}
    \item[4)]\label{Step 4}  Use $\Theta(\log \D)$ new colours to recolour every $e^v$ for which there exists a monochromatic edge of the form $e^vf^w$ without creating new monochromatic edges.    
\end{enumerate}

Step \mref1 is implemented by invoking the \texttt{Greedy Colouring Algorithm}. Therefore, we only need to justify the existence of transversals in Step \mref2, as well as the fact that $\Theta(\log \D)$ of new colours suffice for the recolouring in Step \mref4.

 Before going into the proof of Step \mref2 (which is based on the Lovász Local Lemma) we need to prove \cref{lem:probability of no transversal}, stated below. A version of \cref{lem:probability of no transversal} where $r$ is taken to be at least $5\log k+20$ and all $F_1,...,F_k$ equal the empty set was proven in~\cite{alon1992star} (see Lemma 2.5 therein). The two proofs are very similar and are based on verifying Hall's condition.
\begin{lemma}\label{lem:probability of no transversal}
    There exists an integer $k_0$ such that, for all $k\geq k_0$ and every integer~$r$ satisfying $7\log k \leq r\leq k$, the following holds:
    Let $S_1,\dots, S_k$ be independent random subsets of $[k+r]$, each of which is generated by sampling $r$ elements from $[k+r]$ uniformly and independently at random with replacement.
    Furthermore, let $F_1,\dots,F_k$ be arbitrary fixed subsets of $[k+r]$ of size two. Then the probability that the family of sets \mbox{$\{S_1\setminus F_1,\dots,S_k\setminus F_k\}$} does not have a transversal is at most $k^{1-\frac{r}{5}}$.
\end{lemma}
Note that the number $5$ in the exponent is not optimal. A more careful analysis can lead to a better constant.
\begin{proof}
Our goal is to show that the family $\{S_1\setminus F_1,\dots,S_k\setminus F_k\}$ violates Hall's condition, and therefore has no transversal, with probability at most~$k^{1-\frac{r}{5}}$.
 For $j\in[k]$, let $P_j$ be the probability that there is a set $J\subs[k]$ of size $j$ with $\abs{\bigcup_{i\in J}S_i\setminus F_i}<\abs J$. Our aim is to show that $P_j\leq k^{-\frac{r}{5}}$ for each $j$, which implies that Hall's Theorem is violated with probability at most $\sum_{j=1}^k P_j\leq k^{1-\frac{r}{5}}$. 
 We have  \[P_j\le\binom{k}{j}\binom{k+r}{j}\parens*{\frac{j+2}{k+r}}^{rj}\le \binom{k+r}{j}^2\parens*{\frac{j+2}{k+r}}^{rj},\]
    since there are $\binom{k}{j}$ ways to choose $J\subs[k]$ with $|J|=j$, $\binom{k+r}{j}$ ways to pick a subset $S\subs[k+r]$ of size $j$, and $\parens*{\frac{j+2}{k+r}}^{rj}$ is an upper bound for the probability that $\bigcup_{i\in J}S_i\setminus F_i \subs S$ (this implies that $\abs{\bigcup_{i\in J}S_i\setminus F_i}\leq \abs J$, which contains the event $\abs{\bigcup_{i\in J}S_i\setminus F_i}< \abs J$).

    In order to study this quantity we need to distinguish three cases. In all cases, $k$ will be sufficiently large. 

    \begin{enumerate}
    \item[\textbf{Case 1:}]
    $\frac{k+r}{2}\le j\le k$. Then,
    \begin{align*}
        P_j&\le \binom{k+r}{k+r-j}^2\(1-\frac{k+r-j-2}{k+r}\)^{rj}\\
        &\le (k+r)^{2(k+r-j)}\exp\parens*{-\frac{rj(k+r-j-2)}{k+r}}\\
        &= \exp\Big((k+r-j)\Big(2\log(k+r)-\frac{k+r-j-2}{k+r-j}\cdot \frac{rj}{k+r}\Big)\Big)\\
        &\le \exp\Big((k+r-j)\Big(2\log(k+k)-\frac{2}{3}\cdot \frac{r}{2}\Big)\Big)\\
        &\le \exp\Big((k+r-j)\Big(2\log k + 2 - \frac{r}{3}\Big)\Big).\\
\end{align*}
For the second inequality we used the inequalities $\binom{a}{b}\le a^b$ and $1-x\leq e^{-x}$. Furthermore, we used the assumptions that   $\frac{k+r}{2}\le j\le k$ and $k\geq k_0$ is sufficiently large to deduce that $\frac{j}{k+r} \ge \frac1{2}$ and $ \frac{k+r-j-2}{k+r-j} \ge \frac{2}{3}$, used in the fourth inequality.
Since $r\geq 7\log k$ and $k\geq k_0$ is sufficiently large, the above quantity is bounded above by
        \begin{align*}
 \exp\parens*{-(k+r-j)\parens*{\frac{1}{4}\log k - 2}}
        \le \exp\Big(-\frac{r}{5}\log k\Big).
    \end{align*}    

\item[\textbf{Case 2:}] $\log k\le j\le \frac{k+r}{2}$. In this case we have
    \begin{align*}
        P_j&\le \binom{k+r}{j}^2\(\frac{j+2}{k+r}\)^{rj}
        \le \(\frac{e(k+r)}{j}\)^{2j}\(\frac{1.1\cdot j}{k+r}\)^{rj}\\
    &= \(e^2 \cdot 1.1^r \(\frac{j}{k+r}\)^{r-2}\)^j 
    \le \(e^2 \cdot  1.1^r \(\frac{1}{2}\)^{\frac{r}{2}}\)^{\log k}
        \le\exp\Big(-\frac{r}{5}\log k\Big).
    \end{align*}
For the second inequality, we used that $\binom{n}{x}\leq \left(\frac{en}{x}\right)^x$ for every integer $1\le x \le n$. We used the assumptions that $\log k\le j\le \frac{k+r}{2}$ and $k\geq k_0$ is sufficiently large to deduce that $j+2\le 1.1j$ and $\frac{j}{k+r}\leq \frac{1}{2}$ in the second and third inequalities respectively. For the last inequality, we used that $1.1^r\left(\frac{1}{2}\right)^{\frac{r}{2}}\le \exp\left(-\frac{r}{4}\right)$ holds for all positive $r$. 
    \item[\textbf{Case 3:}] $1\le j\le\log k$. Then  
$$\frac{3j}{k+r}\leq \frac{3\log k}{k}\leq \frac{1}{\sqrt{k}}.$$
Hence,
    \begin{align*}
        P_j&\le \binom{k+r}{j}^2\(\frac{j+2}{k+r}\)^{rj}
        \le \(\frac{e(k+r)}{j}\)^{2j}\(\frac{3\cdot j}{k+r}\)^{rj}\\
        & = \(e^23^2 \(\frac{3j}{k+r}\)^{r-2}\)^j \le \(100 \(\frac{1}{\sqrt{k}}\)^{r-2}\)^j \leq \parens*{\frac{1}{\sqrt{k}}}^{r-3} \\
 &      = \exp\parens*{-\frac12(r-3)\log k} \le\exp\Big(-\frac{r}{5}\log k\Big).
    \end{align*}
In the second inequality,  we again used that $\binom{n}{x}\leq \left(\frac{en}{x}\right)^x$.

\end{enumerate}  
    This shows that, for each $j\in[k]$, we have $P_j\le k^{-\frac{r}{5}}$, as needed.
\end{proof}

We are now ready to prove \cref{theo:X of G33}.

\begin{proof}[Proof of \Cref{theo:X of G33}.]
    Let $\D$ be large enough. We will start by defining a (not necessarily proper)  colouring of $\G$ using at most $\D+7\log \D$ colours. Let $V:=V(G)$ be the set of branch vertices of $\G$. We start by taking a proper colouring of $G$ with $\Delta+1$ colours, $c\colon V(G)\to [\D+1]$, which exists by the \texttt{Greedy Colouring Algorithm}.
    Such a colouring defines a colouring on the branch vertices of $\G$ in  which two incident branch vertices have different colours.
    
    We continue with Step \mref2, which is the content of the following claim:
    \begin{claim}\label{claim 1} There exists an assignment of a list $L_v\subs[\D+7\log \D]$ of size $7\log\D$ to each vertex $v\in V$ such that, for each $w\in V$, the set family $\set{L_v\setminus\set{c(v),c(w)}: v\in N_G(w)}$ has a transversal $T_w$.
    \end{claim}
    \begin{proof}[Proof of Claim \ref{claim 1}.]
        For each $v\in V$, generate $L_v$ randomly by performing $7\log\D$ independent uniform samplings from $[\D+7\log \D]$. Let $B_w$ be the bad event that the family $\set{L_v\setminus\set{c(v),c(w)}:v\in N_G(w)}$ does not have a transversal. By \Cref{lem:probability of no transversal}, we have $\Pr(B_w)\le \D^{1-\frac{7\log\D}{5}}$. Furthermore, $B_w$ is independent of all but at most $\D^2$ other $B_v$'s, namely those corresponding to vertices at distance at most two from $w$ in $G$. Hence, the dependency graph has degree at most $\D^2$ and $\Pr(B_w)\le 1/(4\D^2)$ for $\D$ large enough. Hence, applying the Lovász Local Lemma (\Cref{LLL}) gives that the probability that no bad event happens is positive. In particular, the required list assignment exists.
    \end{proof}

    Let $\{L_v\}_{v\in V}$ be a collection of lists such that all the transversals $T_w$ exist, guaranteed by Claim \ref{claim 1}. Now we extend the colouring $c$ to the inner vertices: For each edge $e=vw$, let $c(e^v)$ be the transversal element of $T_v$ corresponding to the set $L_w\setminus\set{c(v),c(w)}$. This colouring is not necessarily proper, but note that there cannot be a monochromatic edge between a branch vertex and an inner vertex because we excluded $\set{c(v),c(w)}$ from $L_w$ in $T_v$. Furthermore, an edge of the form $e_1^ve_2^v$ can also not be monochromatic since these two colours come from the same transversal $T_v$.

    The only conflicts that remain are those between two inner vertices of the form $e^v$ and $f^w$ with $v\ne w$. Since they are connected in $\G$, they must be of distance at most three in $G^{\frac{1}{3}}$. Unless $f=e=vw$, this shortest path must pass through either $v$ or $w$. We call this branch vertex the \emph{corresponding branch vertex} for the conflict $(e^v,f^w)$. If $f=e=vw$, we choose $v$ or $w$ arbitrarily and call it the corresponding branch vertex. 

    Now consider a conflicting pair $(e^v,f^w)$ of inner vertices and assume, without loss of generality, that $v$ is the corresponding branch vertex. It follows that $f=vw$ and $c(e^v)=c(f^w)\in L_v$. Therefore, if for each $v\in V$ we properly recolour all $e^v$ with $c(e^v)\in L_v$ with new colours, then we have found a proper colouring of $\G$. Let $I_v$ be the set of all inner vertices $e^v$ satisfying $c(e^v)\in L_v$. Observe that, as there are no conflicts of the form $(e^v,f^w)$ with $v= w$,  all the colours $c(e^v)$ for $e^v\in I_v$ are distinct and contained in $L_v$. Thus $|I_v|\leq |L_v|$.

    \begin{claim}\label{claim 2} There is a proper colouring of $\bigcup_{v\in V}I_v$ using at most $21\log\D$ new colours.
    \end{claim}
In the proof of Claim \ref{claim 2} we use the following auxiliary lemma, taken from \cite{Guiduli}.
\begin{lemma}[Lemma 3.2 in \cite{Guiduli}]\label{lem:aux}
    Let $D$ be a directed graph (with possible multiple edges) such that every vertex has indegree at most c. Then $dst(D) \leq 3c$.
\end{lemma}
    
    \begin{proof}[Proof of Claim \ref{claim 2}.]         
    Let $S= \bigcup_{x\in V(G)}I_x$ be the set of vertices in $\G$ to be recoloured with new colours.  Consider the directed graph $D$ with $V(D)=V(G)$ and $E(D)=\set{(w,v):e=vw, e^v\in S}$. Now each vertex in $e^v \in S$ corresponds to a directed edge in $D$. Observe also that,
    for each $w\in V(D)$,  the set of edges directed away from $w$ corresponds to a subset of $S$ that is an independent set in $\G$. Similarly, a directed star forest where the edges of each star are directed away from its centre in $D$ corresponds to an  independent set in $\G$. Furthermore, the maximum indegree in $D$ is at most $\max_{v\in V}\abs{I_v}\le\max_{v\in V}\abs{L_v}\le 7\log \D$.
        By \cref{lem:aux}
        , this implies that $D$ can be partitioned into at most $21\log \D$ directed star forests where the edges of each star are directed away from its centre. By using a different new colour for each of the corresponding independent sets in $\G$, we obtain the required colouring.
    \end{proof}

    After introducing these $21 \log \Delta$ new colours, we are  using $\Delta+28 \log \Delta$ colours in total to properly colour$\G$, as claimed.
\end{proof}

\section{Proof of Theorem \ref{thm:fin} and further comments}\label{s:thm1.3}

In this paper, we obtained asymptotically tight bounds for $\chi(\G)$ as the maximum degree of $G$ grows. The method used to obtain upper bounds for $\chi(\G)$  can be extended to yield upper bounds for $\chi(G^{\frac{k}{k}})$.
Specifically, for every integer $k\geq 2$ there exist a constant $C_k$ such that for every graph $G$ the following holds:
$$\bigg\lfloor \frac{k}{2} \bigg\rfloor \Delta(G)\leq \chi(G^{\frac{k}{k}}) \leq  \bigg\lfloor \frac{k}{2}\bigg \rfloor \Delta(G) + C_k\log \Delta(G).$$

The lower bound of $\lfloor k/2 \rfloor \Delta(G)$ comes from the cliques that are `centred' around branch vertices of maximum degree.
For the upper bound on  $\chi(G^{\frac{k}{k}})$, we follow the same four-step strategy as for $\chi(\G)$, which slightly differs depending on the parity of $k$ in the first two steps. If $k$ is even, then in Step \mref1 we use a proper colouring with $\Delta(G)+1$ colours to colour the branch vertices of $G^{\frac{k}{k}}$. Otherwise, if $k$ is odd, then in Step \mref1, we use a total colouring of $G$ with $\Delta(G)+O(1)$ colours, which is known to exist by~\cite{molloyReed98}, to colour the branch vertices along with the middle vertices on the subdivided edges.

In Step \mref2, we assign  a list $L_v\subseteq [ \lfloor \frac{k}{2}\rfloor \Delta(G) +C'_k\log \Delta(G)]$ of size $\Theta(\log \Delta(G) )$ to every vertex $v$. For a branch vertex $w$ and an edge $e=wv$ of $G$, we will colour the inner vertices of $e$ that are strictly closer to $w$ than to $v$  with a colour from the list $L_v$. For that, we will require the   family $\{L_v\}_{v\in V(G)}$ to have the following property: for each branch vertex $w$, there exist (simple) subsets $L_{v,w}\subseteq L_v\setminus F_{v,w}$  of size $\lfloor k/2 \rfloor$ for each $ v\in N_G(w)$ such that no element appears in two different sets $L_{v,w}$ and $L_{v',w}$. Here $F_{v,w}$ has size at most three and contains $c(v),c(w)$, and, when $k$ is odd, the colour of the middle vertex of the subdivided edge $vw$. 
To study the likelihood of the existence of these subsets of $\{L_v\}_{v\in V(G)}$ we appeal to the generalised Hall's condition in place of Hall's condition, that is, $|N(S)|\geq \lfloor k/2 \rfloor|S|$ (see \cite[Corollary 1.2]{aharoni2000hall}), for every set $S$. Steps \mref3 and \mref4 are identical.

\medskip
A more complicated problem arises when dealing with fractions $\frac{r}{s}$ when $r$ is greater than $s$. In this situation, we must use more colours than the ones used by Brooks' Theorem, as the colouring of a specific  branch vertex may influence not only their neighbours in $G$, but vertices at a higher distance.

\paragraph{Acknowledgments}  This work was initiated at the annual workshop of the Combinatorics and Graph Theory group of Freie Universit\"at Berlin in Wilhelmsaue in September 2023. The authors would like to thank the institution for enabling this research. Finally, the fourth author would like to thank Tibor Szab\'o and the Combinatorics and Graph Theory group at Freie Universit\"at Berlin
for their hospitality during the research visit. Additionally, we thank Moharram Iradmusa for bringing the papers~\cite{hartke2012coloring,iradmusa2010colorings} to our attention. Finally, we thank the anonymous referees for their suggestions on the manuscript, which have improved the quality of the document.

M.A.: This project has received funding from the European Union’s Horizon 2020 research and innovation
programme under the Marie Sk\l{}odowska-Curie grant agreement No 101034413
\includegraphics[width=5.5mm, height=4mm]{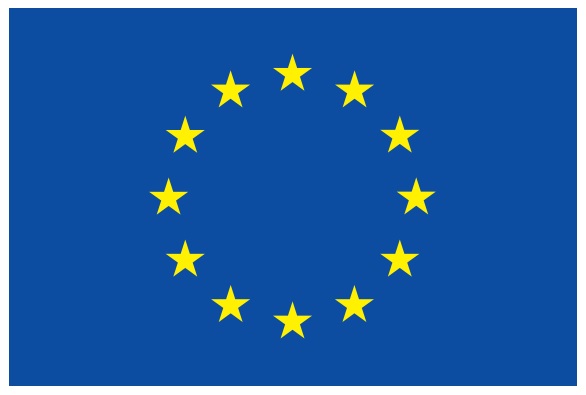}.

S.B.: The research leading to these results was supported by EPSRC, grant no.\ EP/V048287/1. There are no additional data beyond that contained within the main manuscript.

S.R.: Funded by the Deutsche Forschungsgemeinschaft (DFG, German Research Foundation) under Germany´s Excellence Strategy – The Berlin Mathematics Research Center MATH+ (EXC-2046/1, project ID: 390685689).

J.R.  acknowledges the support of the Grant PID2020-113082GB-I00 funded by MICIU/AEI/10.13039/501100011033, and the Severo Ochoa and María de Maeztu Program for Centers and Units of Excellence in R\&D (CEX2020-001084-M).

\bibliographystyle{amsplain}
\bibliography{bib}

\providecommand{\bysame}{\leavevmode\hbox to3em{\hrulefill}\thinspace}
\providecommand{\MR}{\relax\ifhmode\unskip\space\fi MR }
\providecommand{\MRhref}[2]{%
  \href{http://www.ams.org/mathscinet-getitem?mr=#1}{#2}
}
\providecommand{\href}[2]{#2}
\begin{thebibliography}{10}

\bibitem{aharoni2000hall}
Ron Aharoni and Penny Haxell, \emph{Hall's theorem for hypergraphs}, J. Graph
  Theory \textbf{35} (2000), no.~2, 83--88.

\bibitem{alon1992star}
Noga Alon, Colin McDiarmid, and Bruce Reed, \emph{Star arboricity},
  Combinatorica \textbf{12} (1992), 375--380.

\bibitem{behzad1965graphs}
Mehdi Behzad, \emph{Graphs and their chromatic numbers}, Ph.D. thesis, Michigan
  State University, 1965.

\bibitem{Guiduli}
Barry Guiduli, \emph{On incidence coloring and star arboricity of graphs},
  Discrete Math. \textbf{163} (1997), no.~1-3, 275--278.

\bibitem{hartke2012coloring}
Stephen Hartke, Hong Liu, and {\v{S}}{\'a}rka Pet{\v{r}}{\'\i}{\v{c}}kov{\'a},
  \emph{On coloring of fractional powers of graphs}, arXiv:1212.3898 (2012).

\bibitem{hindmolloyreed99}
Hugh Hind, Michael Molloy, and Bruce Reed, \emph{Total coloring with
  {$\Delta+{\rm poly}(\log\Delta)$} colors}, SIAM J. Comput. \textbf{28}
  (1999), no.~3, 816--821.

\bibitem{iradmusa2010colorings}
Moharram~N Iradmusa, \emph{On colorings of graph fractional powers}, Discrete
  Math. \textbf{310} (2010), no.~10-11, 1551--1556.

\bibitem{survey}
Narayanan~Narayanan Jayabalan~Geetha and Kanagasabapathi Somasundaram,
  \emph{Total colorings--a survey}, AKCE International Journal of Graphs and
  Combinatorics \textbf{20} (2023), no.~3, 339--351.

\bibitem{molloyReed98}
Michael Molloy and Bruce Reed, \emph{A bound on the total chromatic number},
  Combinatorica \textbf{18} (1998), no.~2, 241--280.

\bibitem{MolloyReed}
Michael Molloy and Bruce Reed, \emph{Graph colouring and the probabilistic
  method}, Algorithms and Combinatorics, vol.~23, Springer-Verlag, Berlin,
  2002.

\bibitem{mozafari2021note}
Mahsa Mozafari-Nia and Moharram~N. Iradmusa, \emph{A note on coloring of
  $\frac{3}{ 3}$-power of subquartic graphs}, Australas. J. Comb. \textbf{79}
  (2021), no.~3, 454--460.

\bibitem{mozafari2022simultaneous}
Mahsa Mozafari-Nia and Moharram~N. Iradmusa, \emph{Simultaneous coloring of
  vertices and incidences of graphs}, Australas. J. Comb. \textbf{85} (2023),
  no.~3, 287--307.

\bibitem{mozafari2023outerplanar}
Mahsa Mozafari-Nia and Moharram~N. Iradmusa, \emph{Simultaneous coloring of
  vertices and incidences of outerplanar graphs}, Electron. J. Graph Theory
  Appl. (EJGTA) \textbf{11} (2023), no.~1, 245--262.

\bibitem{mozafari2024hypercubes}
Mahsa Mozafari-Nia and Moharram~N. Iradmusa, \emph{Simultaneous coloring of
  vertices and incidences of hypercubes}, Commun. Comb. Optim. \textbf{9}
  (2024), no.~1, 67--77.

\bibitem{vizing64}
Vadim Vizing, \emph{On an estimate of the chromatic class of a p-graph},
  Diskret. Analiz. \textbf{3} (1964), 25--30.

\bibitem{wang2018coloring}
Fang Wang and Xiaoping Liu, \emph{Coloring 3-power of 3-subdivision of subcubic
  graph}, Discrete Math. Algorithms Appl. \textbf{10} (2018), no.~03, 1850041.

\bibitem{yap}
Hian~Poh Yap, \emph{Total colourings of graphs}, Lecture Notes in Mathematics,
  vol. 1623, Springer-Verlag, Berlin, 1996.

\end{thebibliography}

\end{document}